\documentclass[11pt,a4paper]{article}
\usepackage{cmap} 
\usepackage[utf8]{inputenc}			
\usepackage[T1]{fontenc}            
\usepackage{a4wide}
\usepackage{amssymb,amsmath,amsthm,bm}
\usepackage{mathtools}
\usepackage[english]{babel}
\usepackage{tikz}
\usetikzlibrary{arrows}
\usepackage{graphicx}
\usepackage{enumitem}

\usepackage{hyperref}
\usepackage[ruled,vlined]{algorithm2e}

\newtheorem{theorem}{Theorem}
\newtheorem{corollary}[theorem]{Corollary}
\newtheorem{lemma}[theorem]{Lemma}

\newtheorem{claim}[theorem]{Claim}

\newtheorem{obs}[theorem]{Observation}
\newtheorem{defi}[theorem]{Definition}

\newlist{proplist}{enumerate}{3}
\setlist[proplist]{label=(\emph{\roman*}),ref=(\roman*)}

\newcommand\eps{{\varepsilon}}
\newcommand\Dc{{\mathcal{D}}}
\newcommand\Pc{{\mathcal{P}}}

\newcommand\Sc{{\mathcal{S}}}
\newcommand\Hc{{\mathcal{H}}}
\newcommand\Bc{{\mathcal{B}}}
\newcommand\Fc{{\mathcal{F}}}
\newcommand\Gc{{\mathcal{G}}}
\newcommand\Fcal{{\mathcal{F}}}
\newcommand\Ac{{\mathcal{A}}}
\newcommand\Bcal{{\mathcal{B}}}

\newcommand{\Des}{Des}
\newcommand{\Sib}{\Dc_{Sib}}
\newcommand{\Fix}{\Dc_{Des}}
\newcommand{\Pfix}{\Pc_{fix}}

\title{Realizing an $m$-uniform four-chromatic hypergraph with disks}

\author{
 Gábor Damásdi\\\small MTA-ELTE Lendület\\ \small   Combinatorial Geometry Research Group  \and Dömötör Pálvölgyi\\\small MTA-ELTE Lendület\\ \small   Combinatorial Geometry Research Group}
 
\index{Damásdi, Gábor}
\index{Pálvölgyi, Dömötör}

\begin{document}
\thispagestyle{empty}
\maketitle

\begin{abstract}
    We prove that for every $m$ there is a finite point set $\Pc$ in the plane such that no matter how $\Pc$ is three-colored, there is always a disk containing exactly $m$ points, all of the same color.
    This improves a result of Pach, Tardos and T\'oth who proved the same for two colors.
    The main ingredient of the construction is a subconstruction whose points are in convex position. 
    Namely, we show that for every $m$ there is a finite point set $\Pc$ in the plane in convex position such that no matter how $\Pc$ is two-colored, there is always a disk containing exactly $m$ points, all of the same color.
    We also prove that for unit disks no similar construction can work, and several other results.
\end{abstract}

\section{Introduction}

Coloring problems for hypergraphs defined by geometric range spaces have been studied a lot in different settings.
A pair $(\Pc, \Sc)$, where $\Pc$ is a set of points in the plane and $\Sc$ is a family of subsets of the plane (the \emph{range space}), defines a (primal) hypergraph $\Hc(\Pc,\Sc)$ whose vertex set is $\Pc$, and for each $S\in\Sc$ we add the edge $S\cap \Pc$ to the hypergraph. 
Given any hypergraph $\Gc$, a planar realization of $\Gc$ is defined as a pair $(\Pc, \Sc)$ for which $\Hc(\Pc,\Sc)$ is isomorphic to $\Gc$.
If $\Gc$ can be realized with some pair $(\Pc, \Sc)$ where $\Sc$ is from some family $\Fc$, then we say that $\Gc$ is realizable with $\Fc$.

A hypergraph is (properly) $c$-colorable if its vertices can be colored by
$c$ colors such that no edge is monochromatic.
In this paper we focus on the $c$-colorability of hypergraphs realizable with disks.
It is an easy consequence of the properties of Delaunay-triangulations and the Four Color Theorem that any hypergraph realizable with disks is four-colorable if every edge contains at least two vertices.
Since $K_4$ is realizable with disks, this is sharp.
But are less colors sufficient if all edges are required to contain at least $m$ vertices for some large enough constant $m$?
Pach, Tardos and Tóth \cite{MR2364757} have shown that two colors are not enough for any $m$, i.e., 
for any $m$, there exists an $m$-uniform hypergraph that is not two-colorable and that permits a planar realization with disks.
Our main theorem is the following strengthening, 
which shows that three colors are also not enough, by realizing a four-chromatic hypergraph, completely resolving this question.

\begin{theorem}\label{thm:main}
For any $m$, there exists an $m$-uniform hypergraph that is not three-colorable and that permits a planar realization with disks.
\end{theorem}

The proof of Theorem \ref{thm:main} is based on two ideas. The first one is from \cite{abafree}, where colorings of so-called ABAB-free hypergraphs were considered. We will use a construction, based on one from \cite{abafree}, to create a point set that is not two-colorable with respect to disks and, furthermore, the points are close to a prescribed set of points that lie on a circle. Then using ideas from  \cite{MR4012917} and \cite{MR2364757}, we will combine several copies of this non-two-colorable construction to create point sets that are not three-colorable.
The non-two-colorable construction 
can be of independent interest.

\begin{theorem}\label{thm:stabbed}
For any $m$, there exists an $m$-uniform hypergraph that is not two-colorable and that permits a planar realization with disks that all contain some fixed point.\\
Moreover, in the realization $(\Pc, \Dc)$, the points $\Pc$ can be placed arbitrarily close to some given points on a circle such that the boundary of each disk from $\Dc$ is also arbitrarily close to this circle.
\end{theorem}

Previously, such a construction was only known for pseudo-disks containing a fixed point, and it was also shown that such hypergraphs are always three-colorable (already for $m=2$) \cite{MR4012917}.

Note that if we require the points of $\Pc$ to be placed close enough to the given points on the circle, then the points of $\Pc$ will be in convex position.

The construction of Theorem \ref{thm:main} can be generalized as follows.

\begin{theorem}\label{thm:gen_main}

 Let $C$ be any closed convex set in the plane, that has two parallel supporting lines such that $C$ is strictly convex in some neighbourhood of the two points of tangencies. For any $m$, there exists an $m$-uniform hypergraph that is not three-colorable and that permits a planar realization with homothets of $C$.
\end{theorem}


To complement our results, we also show that no similar construction for unit disks exists.

\begin{theorem}\label{thm:unit}
For any $k$, any finite point set $\Pc$ can be $k$-colored such that any unit disk, that contains some fixed point $o\notin \Pc$ 
and $8k-7$ points from $\Pc$, will contain all $k$ colors.
\end{theorem}

Note that the condition $o\notin \Pc$ is in fact not required, we only state our results this way to emphasize that the common point $o$ need not be from $\Pc$.

It is known that without requiring a common point $o$, the statement does not hold \cite{unsplittable}.

By using the well-known equivalence of the hypergraphs defined by primal and dual range spaces of the translates of any set \cite{Pach86,surveycd}, we can conclude the following.

\begin{corollary}
For any $k$, any finite collection of unit disks $\Dc$ can be partitioned into $k$ parts such that any point of a fixed unit disk $D_0\notin\Dc$ that was covered by at least $8k-7$ members of $\Dc$ will be covered by all $k$ parts.
\end{corollary}

This statement is sharp in the sense that for a disk $D_0$ of larger radius it fails already for $k=2$, even if the members of $\Dc$ are required to be very close to each other; this follows from taking the dual of the construction from \cite{unsplittable}. 
The function $8k-7$ is unlikely to be sharp.\\

Theorem \ref{thm:unit} can be extended in two different ways. Firstly, instead of stabbed unit disks we can consider stabbed translates of a given convex set.

\begin{theorem}\label{thm:convex}
For any plane convex set $C$ and any integer $k$ the following holds. Any finite point set $\Pc$ can be $k$-colored such that any translate of $C$, that contains some fixed point $o\notin \Pc$ and $24k-23$ points from $\Pc$, will contain all $k$ colors. 
\end{theorem}

Secondly, it is also interesting to consider the case when our family contains stabbed homothetic copies of a given convex set. Theorem \ref{thm:stabbed} implies that the natural analogue of Theorem \ref{thm:unit} does not hold for homothetic copies of a disk. On the other hand we can show the following for polygons. 

\begin{theorem}\label{thm:polygon}
For any convex polygon $P$ there is a $t_P$ such that for any $k$ the following holds. Any finite point set $\Pc$ can be $k$-colored such that any homothetic copy of $P$, that contains some fixed point $o\notin \Pc$ and $t_P k$ points from $\Pc$, will contain all $k$ colors.
\end{theorem}

The rest of the paper is organized as follows.
In Section \ref{sec:2} we prove Theorem \ref{thm:stabbed}, then building on this construction, in Section \ref{sec:3} we prove Theorem \ref{thm:main}.
In Section \ref{sec:4} we prove Theorem \ref{thm:unit}.
The sketch of the proofs of the generalizations for other shapes (not disks), can be found in Section \ref{sec:5}.
Finally, in Section \ref{sec:6} we highlight the most important problems left open.
We end this introduction with a bit more history, and a basic observation.

\subsubsection*{Related results}
Pach proved in 1986 that any sufficiently thick covering of the plane
\cite{Pach86} by the translates of a centrally symmetric open convex polygon can be partitioned into two disjoint coverings.
The proof followed from showing that for every such polygon $P$ there is an $m(P)$ for which any hypergraph, whose edges contain at least $m(P)$ vertices and can be realized with translates of $P$, is two-colorable.
Several results followed, eventually showing that the similar statement is true for the translates of all convex polygons \cite{MR2812512,PT10}, while counterexamples were given for non-convex polygons \cite{MR2364757,MR2679054} and convex shapes with a smooth boundary \cite{unsplittable}.
We know much less when instead of translates, homothetic (i.e., scaled and translated) copies are considered \cite{homotsquare,MR3151767,MR3216669,kovacs} or about the problem of decomposing into multiple coverings/polychromatic colorings \cite{MR3126347,MR2812512,MR2844088}.
For a summary of the most closely related results, see Table \ref{table}; for more results, see the decade-old survey \cite{surveycd} or the webpage \url{https://coge.elte.hu/cogezoo.html}.

\begin{table}[]
    \begin{center}
		\begin{tabular}{|l|l|l|}
			\hline
			& \textbf{translates} & \textbf{homothets} \\ \hline
			
			\textbf{triangles} & \begin{tabular}[c]{@{}l@{}}$\chi_{fat}=2$ \cite{TT07}\\$m_k=O(k)$ \cite{MR2812512}\end{tabular} &

			\begin{tabular}[c]{@{}l@{}}$\chi_{fat}=2$ \cite{octants}\\$m_k=O(k^{4.09})$ \cite{MR3151767} + \cite{moreoctants}\end{tabular}  \\ \hline
			
			\begin{tabular}[c]{@{}l@{}}
				\textbf{convex}\\ \textbf{polygons}
			\end{tabular} & \begin{tabular}[c]{@{}l@{}}$\chi_{fat}=2$ \cite{PT10}\\$m_k=O(k)$ \cite{MR2812512} \end{tabular} &
			\begin{tabular}[c]{@{}l@{}}
			$2\le \chi_{fat}\le 3$ \cite{3propercol}\\
			$\chi_{fat}=2$ for squares \cite{homotsquare}
		    \end{tabular}  \\ \hline
			
			\begin{tabular}[c]{@{}l@{}}
				\textbf{non-convex}\\ \textbf{polygons$^*$}
			\end{tabular} & \begin{tabular}[c]{@{}l@{}}$3\le \chi_{fat}$ \cite{MR2679054}\end{tabular} & \begin{tabular}[c]{@{}l@{}}$3\le \chi_{fat}$ \cite{MR2679054}\end{tabular} \\ \hline
		
			\begin{tabular}[c]{@{}l@{}}
				\textbf{stabbed}\\ \textbf{convex}\\ \textbf{polygons}
			\end{tabular} & \begin{tabular}[c]{@{}l@{}}$\chi_{fat}=2$ \cite{PT10}\\$m_k=O(k)$ \cite{MR2812512}\end{tabular} & \begin{tabular}[c]{@{}l@{}}$\chi_{fat}=2$ [NEW]\\$m_k=O(k)$ [NEW]\end{tabular} \\ \hline
		
			\begin{tabular}[c]{@{}l@{}}
				\textbf{stabbed}\\ \textbf{disks}
			\end{tabular} & \begin{tabular}[c]{@{}l@{}}$\chi_{fat}=2$ [NEW]\\$m_k=O(k)$ [NEW]\end{tabular} & \begin{tabular}[c]{@{}l@{}}$\chi_{fat}\le 3$ \cite{MR4012917}\\ $\chi_{fat}= 3$
			[NEW]\end{tabular} \\ \hline
		
			\textbf{disks} & \begin{tabular}[c]{@{}l@{}}$3\le \chi_{fat}\le 4$ \cite{unsplittable}\end{tabular} & \begin{tabular}[c]{@{}l@{}}$\chi_{fat}=4$ [NEW] \end{tabular} \\ \hline
			\end{tabular}
			\caption{Summary of old and new results related to the topic of this paper. For a family, $\chi_{fat}$ denotes the smallest number $k$ for which there is an $m$ such that any finite set of points can be $k$-colored such that any member of the family with at least $m$ points will contain at least two colors. For $k$ colors, $m_k$ denotes the smallest number $m$ such that any finite set of points can be $k$-colored such that any member of the family with at least $m$ points will contain all $k$ colors. Results from this paper are marked with `[NEW]'.
			\newline $^*$ There are some very special non-convex polygons for which $\chi_{fat}=2$---for the complete classification, see \cite{PT10} or \cite{surveycd}.}\label{table}
    \end{center}		
\end{table}

The above papers mainly focused on two- or polychromatic colorability.
Proper three-colorability of geometric hypergraphs was studied in detail in \cite{wcf2} and \cite{3propercol}.
In the latter paper it was shown that for every convex polygon $P$ there is an $m(P)$ such that every $m(P)$-uniform hypergraph realizable by homothetic copies of $P$ is proper three-colorable.
Our results imply that this is not the case for disks, disproving a conjecture from both of the above papers.\\

In this paper all disks are assumed to be open.\footnote{But note that our results also hold for closed disks, as we could slightly shrink any finite system of open disks to contain the same points of a finite point set.} Let $\Pc$ be a point set and let $\Dc$ be a family of disks.  An important folklore observation that we will use many times is that small perturbations of the points and the disks will not change the hypergraph $\Hc(\Pc,\Dc)$. To put this into more precise terms, we will say that two points are \emph{$\eps$-close} if their distance is less than $\eps$ and two disks/circles are \emph{$\eps$-close} if their centers are $\eps$-close and the difference of their radius is also smaller than $\eps$.
    
\begin{obs}\label{obs:pert}
    Suppose we have a finite point set $\Pc$ and a finite set of open disks $\Dc$ such that none of the points lie on the boundary of any of the disks. Then there is an $\eps$ such that replacing each disk with any $\eps$-close disk and each point with any $\eps$-close point will not change the hypergraph $\Hc(\Pc,\Dc)$.    
\end{obs}


\section{A point set that is not two-colorable}\label{sec:2}

Here we prove Theorem \ref{thm:stabbed} by realizing for any $m$ a non-two-colorable $m$-uniform hypergraph with disks that all contain some fixed point.

Our main lemma is the following. Combined with Observation \ref{obs:pert}, this gives us a way to perturb the points of $\Pc$ with changing only a small part of the hypergraph $\Hc(\Pc,\Dc)$.

\begin{lemma}\label{lemma:step}
 If $\eps >0$, $C$ is a circle, and  $a,b_1,\dots, b_t, c$ are points on $C$ in this order, then there is a circle $C'$ and points $b'_1,\dots, b'_t$ on $C'$ with the following properties.
 \begin{enumerate}
     \item $C'$ is $\eps$-close to $C$.
     \item $b'_i$ is $\eps$-close to $b_i$ for each $i\in[t]$.
     \item $C'$ intersects $C$ between $a$ and $b_1$, and between $b_t$ and $c$.
     \item Each $b'_i$ is outside of $C$.
 \end{enumerate}
\end{lemma}

\begin{figure}[!ht]
    \centering
    \scalebox{0.8}{\definecolor{uuuuuu}{rgb}{0.26666666666666666,0.26666666666666666,0.26666666666666666}
\begin{tikzpicture}[line cap=round,line join=round,>=triangle 45,x=1.0cm,y=1.0cm,scale=1.6]
\clip(0.563065315586993,-6.0) rectangle (9.606509698847772,-2.1845350799150536);
\fill [line width=0.0pt,fill=black,fill opacity=0.20000000298023224] (3.7082361104981363,-3.214323567210344) circle (0.402885346005847cm);
\fill [line width=0.0pt,fill=black,fill opacity=0.20000000298023224] (4.7401353502623875,-3.0084501301102664) circle (0.46277116788104733cm);
\fill [line width=0.0pt,fill=black,fill opacity=0.20000000298023224] (5.970142500145331,-3.1194299994186725) circle (0.46179720183543277cm);
\draw [line width=1.5pt] (5.,-7.) circle (4.cm);
\draw [line width=1.5pt] (4.999529754525854,-5.678344770003286) circle (2.9278067948682214cm);
\begin{normalsize}
\draw (2.228201319308865,-3.4611419917762163) node[anchor=north west] {$a$};
\draw (2.6362205930300467,-3.1200638092178905) node[anchor=north west] {$A$};
\draw (7.00457042047425,-3.149299082008604) node[anchor=north west] {$B$};
\draw (7.4405504216373345,-3.4221616280552647) node[anchor=north west] {$c$};
\draw (3.515827867449079,-3.2564950822412206) node[anchor=north west] {$b_1$};
\draw (4.380591869542632,-3.012867808985274) node[anchor=north west] {$b_2$};
\draw (5.921080963729508,-3.1603187182876524) node[anchor=north west] {$b_3$};
\draw (3.447612230937414,-2.3891819894500492) node[anchor=north west] {$b_1'$};
\draw (4.578042778845011,-2.0968292615429127) node[anchor=north west] {$b_2'$};
\draw (6.115218055124979,-2.411221262008146) node[anchor=north west] {$b_3'$};
\draw (7.34,-4.630552903404762) node[anchor=north west] {$C'$};
\draw (8.44684387814946,-4.464886357590718) node[anchor=north west] {$C$};
\draw [line width=1.5pt] (4.999529754525854,-5.678344770003286)-- (3.640498325065157,-3.0850676537246313);
\draw [line width=1.5pt] (4.999529754525854,-5.678344770003286)-- (4.716410881102014,-2.7642589216474525);
\draw [line width=1.5pt] (4.999529754525854,-5.678344770003286)-- (6.037879386681707,-2.9408490769512707);
\draw [line width=1.5pt,dash pattern=on 3pt off 3pt] (4.998488039730598,-2.7505381604563346)-- (5.,-7.);
\draw (5.011473689310475,-5.629650359684269) node[anchor=north west] {$O$};
\draw [fill=black] (5.,-7.) circle (2.0pt);
\draw [fill=black] (2.5712307984019978,-3.821780346582539) circle (2.0pt);
\draw [fill=white] (3.0106220029263135,-3.529787443864701) circle (1.0pt);
\draw [fill=black] (3.7082361104981363,-3.214323567210344) circle (2.0pt);
\draw [fill=black] (4.7401353502623875,-3.0084501301102664) circle (2.0pt);
\draw [fill=black] (5.970142500145331,-3.1194299994186725) circle (2.0pt);
\draw [fill=white] (6.986908087591941,-3.528372679641184) circle (1.0pt);
\draw [fill=black] (7.412948270683936,-3.809752259932874) circle (2.0pt);
\draw [fill=black] (4.999529754525854,-5.678344770003286) circle (2.0pt);
\draw [fill=black] (3.640498325065157,-3.0850676537246313) circle (2.0pt);
\draw [fill=black] (4.716410881102014,-2.7642589216474525) circle (2.0pt);
\draw [fill=black] (6.037879386681707,-2.9408490769512707) circle (2.0pt);
\end{normalsize}
\end{tikzpicture}}
    \caption{Lemma \ref{lemma:step}.}
    \label{fig:step}
\end{figure}

\begin{proof}
Choose points $A$ and $B$ between between $a$ and $b_1$, and between $b_k$ and $c$, respectively (see Figure \ref{fig:step}). Choose $O$ on the perpendicular bisector of $AB$, close to the center of $C$. $C'$ will be the circle centered at $O$ passing through $A$ and $B$. Project $b_1,\dots, b_t$ onto $C'$ using $O$ as center. If $O$ is close enough to the center of $C$, this will clearly satisfy the requirements.    
\end{proof}

\subsection{Hypergraphs based on rooted trees}

The following hypergraph construction was used in \cite{MR2364757} to create several counterexamples for coloring problems.  

\begin{defi}
 For any rooted tree $T$, let $\Hc(T)$ denote the hypergraph on vertex set $V(T)$, whose hyperedges are all sets of the following two types.
 
 \begin{enumerate}
     \item Sibling hyperedges: for each vertex $v\in V(T)$ that is not a leaf, take the set
$S(v)$ of all children of $v$.
      \item Descendent hyperedges: for each leaf $v\in V(T)$, take all vertices along the unique path $Q(v)$ from the root to $v$.
 \end{enumerate}
\end{defi}

It is easy to see that $\Hc(T)$ is not two-colorable for any $T$. Either there is a monochromatic sibling edge, or we can follow the color of the root down to a leaf, finding a monochromatic descendent edge.  We can create an $m$-uniform hypergraph by choosing $T$ to be the complete $m$-ary tree of depth $m$. The non-two-colorable construction of Pach, Tardos and Tóth is also based on these hypergraphs.

\begin{theorem}[Pach, Tardos and Tóth \cite{MR2364757}]\label{thm:ptt}
 For any rooted tree $T$, the hypergraph $\Hc(T)$ permits a planar realization $(\Pc, \Dc)$ with disks in general position such that every disk $D\in \Dc$ has
a point on its boundary that does not belong to the closure of any other disk $D'\in \Dc$. \end{theorem}
 
 In order to be able to later build a point set that is not three-colorable, we will first extend Theorem \ref{thm:ptt} by showing that we can require the points to be close to a prescribed set of concyclic points, and require the disks to be close to the circle that contains the prescribed points. (We loose the property that every disk $D\in \Dc$ has a point on its boundary that does not belong to the closure of any other disk $D'\in \Dc$, so strict speaking Theorem \ref{thm:construction} is not a generalization of Theorem \ref{thm:ptt}.)  
 
 \begin{theorem}\label{thm:construction}
 If $\gamma>0$, $C$ is a circle  and $q_1,q_2,\dots, q_{n}$ are distinct points on $C$, then for any rooted tree $T$ on $n$ vertices, the hypergraph $\Hc(T)$ permits a planar realization $(\Pc, \Dc)$ with disks such that 
 \begin{enumerate}
 \renewcommand{\labelenumi}{(\Roman{enumi})}
    \item  $\Pc=\{p_1,\dots, p_n\}$, and each $p_i$ is $\gamma$-close to $q_i$. 
    \item  Every $D\in \Dc$ is $\gamma$-close to $C$.
 \end{enumerate}
  \end{theorem}
 
 An important property of rooted trees is that we can order their vertices in a special way. For a vertex $v$, let $\Des(v)$ denote all descendants of $v$. Keszegh and Pálvölgyi \cite{abafree} have shown that there is an ordering on the vertices of $T$ such that 
 \begin{enumerate}
    \item For each vertex $v\in V(T)$ the vertices in $S(v)$ are consecutive and they appear in the order later than $v$. 
    \item  Furthermore, suppose $S(v)=\{r_1,\dots, r_k\}$ and they are in this order. Then the vertices  $r_1, \dots,r_{k-1}, r_k, \Des(r_k), \Des(r_{k-1}),\dots, \Des(r_1)$ are ordered like this, and the rest of the vertices of $T$ are not in this interval. (The internal order of each $Des(r_{i})$ is not specified by this statement.) 
\end{enumerate}

Call an order satisfying these properties a \emph{siblings first order} \cite{MR4012917} of $T$.

 \begin{proof}[Proof of Theorem \ref{thm:construction}]  
      In the planar realization of $\Hc(T)$, the vertices will correspond to the points $q_i$ according to an (arbitrary) fix siblings first order.
      We start by showing that the sibling hyperedges can be easily realized and we only need to consider the descendent hyperedges.
      
\subsubsection*{Sibling hyperedges}
      From the properties of the siblings first order we know that for each $v$ the vertices of $S(v)$ are consecutive, i.e., the  points $q_i$ corresponding to $S(v)$ are neighbours along the circle $C$. We apply Lemma \ref{lemma:step} to find a disk that is $\gamma$-close to $C$, and contains exactly the points of $S(v)$. We can also ensure that no point $q_i$ lies on the boundary of this disk. We repeat this for each $v\in V(\Hc)$, until each sibling hyperedge is realized. Let $\Sib$ denote the set of these disks. We apply Observation \ref{obs:pert} to $(\{q_1,\dots, q_n\},\Sib)$ to get $\eps_{Sib}$. That is, if each point $p_i$ is $\eps_{Sib}$-close to $q_i$, then $(\Pc,\Sib)$ will still represent the sibling hyperedges. Therefore, it is enough to show that we can realize the descendent hyperedges for any $\gamma$. 
      
\subsubsection*{Descendent hyperedges}

        It is useful to realize a slightly extended hypergraph. In $\Hc(T)$, we have a descendent hyperedge for each leaf. Now we will create a descendent hyperedge for non-leaf vertices too. So let $Q(v)$ denote the path from the root to $v$, and for each vertex $v$, we will realize the hyperedge $Q(v)$. The disk realizing $Q(v)$ will be denoted by $B(v)$. Let this extended hypergraph be denoted by $\Hc'(T)$. We will realize not only $\Hc(T)$, but $\Hc'(T)$. See Figure \ref{fig:tree} for an example where we omitted the sibling edges.    
        
        \begin{figure}[!ht]
        \centering
        \scalebox{1.0}{\definecolor{ududff}{rgb}{0.,0.,0.}
\definecolor{ffqqqq}{rgb}{0.,0.,0.}
\definecolor{cadmiumgreen}{rgb}{0.0, 0.42, 0.24}

\begin{tikzpicture}[line cap=round,line join=round,>=triangle 45,x=1.0cm,y=1.0cm,scale=0.4]
\clip(-14.367616979866408,-4.957139708737502) rectangle (15.132910466027337,8.31944963675636);
\draw [line width=1.5pt,color=black] (8.,2.) circle (6.cm);
\draw [line width=1.5pt,color=brown] (5.802116387755977,0.1362029631372868) circle (4.205004262810615cm);
\draw [line width=1.5pt,color=green] (4.683143455467177,0.13148071717624432) circle (3.6595353060840554cm);
\draw [line width=1.5pt,color=blue] (5.3184197418610015,-1.3220764596806789) circle (3.4832689814465416cm);
\draw [line width=1.5pt,color=orange] (10.197883612244024,0.13620296313728653) circle (4.205004262810615cm);
\draw [line width=1.5pt,color=cadmiumgreen] (11.316856544532822,0.1314807171762439) circle (3.6595353060840554cm);
\draw [line width=1.5pt,color=red] (10.681580258138998,-1.322076459680679) circle (3.4832689814465416cm);
\draw [line width=1.5pt] (8.,0.5)-- (2.8002124687001437,-1.898112637247945);
\draw [line width=1.5pt] (1.3196487711664556,0.1624532545214694)-- (2.8002124687001437,-1.898112637247945);
\draw [line width=1.5pt] (4.378574322913053,-4.3236568755059565)-- (2.8002124687001437,-1.898112637247945);
\draw [line width=1.5pt] (8.,0.5)-- (13.199787531299856,-1.8981126372479458);
\draw [line width=1.5pt] (11.621425677086947,-4.3236568755059565)-- (13.199787531299856,-1.8981126372479458);
\draw [line width=1.5pt] (14.680351228833543,0.16245325452146855)-- (13.199787531299856,-1.8981126372479458);
\draw [line width=1.5pt] (-8.885971702069108,1.935135150858204)-- (-11.885971702069108,-0.06486484914179613);
\draw [line width=1.5pt] (-5.885971702069108,-0.06486484914179613)-- (-8.885971702069108,1.935135150858204);
\draw [line width=1.5pt] (-5.885971702069108,-0.06486484914179613)-- (-7.885971702069108,-2.064864849141795);
\draw [line width=1.5pt] (-5.885971702069108,-0.06486484914179613)-- (-3.8859717020691082,-2.064864849141795);
\draw [line width=1.5pt] (-11.885971702069108,-0.06486484914179613)-- (-13.885971702069108,-2.064864849141795);
\draw [line width=1.5pt] (-11.885971702069108,-0.06486484914179613)-- (-9.885971702069108,-2.064864849141795);
\begin{scriptsize}
\draw [fill=ffqqqq,color=red] (8.,0.5) circle (6.pt);
\draw [fill=ffqqqq,color=brown] (2.8002124687001437,-1.898112637247945) circle (6.pt);
\draw [fill=ffqqqq,color=green] (1.3196487711664556,0.1624532545214694) circle (6.pt);
\draw [fill=ffqqqq,color=blue] (4.378574322913053,-4.3236568755059565) circle (6.pt);
\draw [fill=ffqqqq,color=black] (8.,0.5) circle (6.pt);
\draw [fill=ffqqqq,color=blue] (13.199787531299856,-1.8981126372479458) circle (6.pt);
\draw [fill=ffqqqq,color=red] (11.621425677086947,-4.3236568755059565) circle (6.pt);
\draw [fill=ffqqqq,color=red] (13.199787531299856,-1.8981126372479458) circle (6.pt);
\draw [fill=ffqqqq,color=cadmiumgreen] (14.680351228833543,0.16245325452146855) circle (6.pt);
\draw [fill=ffqqqq,color=orange] (13.199787531299856,-1.8981126372479458) circle (6.pt);
\draw [fill=ududff,color=black] (-8.885971702069108,1.935135150858204) circle (6.pt);
\draw [fill=ududff,color=brown] (-11.885971702069108,-0.06486484914179613) circle (6.pt);
\draw [fill=ududff,color=orange] (-5.885971702069108,-0.06486484914179613) circle (6.pt);
\draw [fill=ududff,color=red] (-7.885971702069108,-2.064864849141795) circle (6.pt);
\draw [fill=ududff,color=cadmiumgreen] (-3.8859717020691082,-2.064864849141795) circle (6.pt);
\draw [fill=ududff,color=green] (-13.885971702069108,-2.064864849141795) circle (6.pt);
\draw [fill=ududff,color=blue] (-9.885971702069108,-2.064864849141795) circle (6.pt);
\end{scriptsize}
\end{tikzpicture}}
        \caption{Disks realizing all paths from the root.} 
        \label{fig:tree}
    \end{figure}
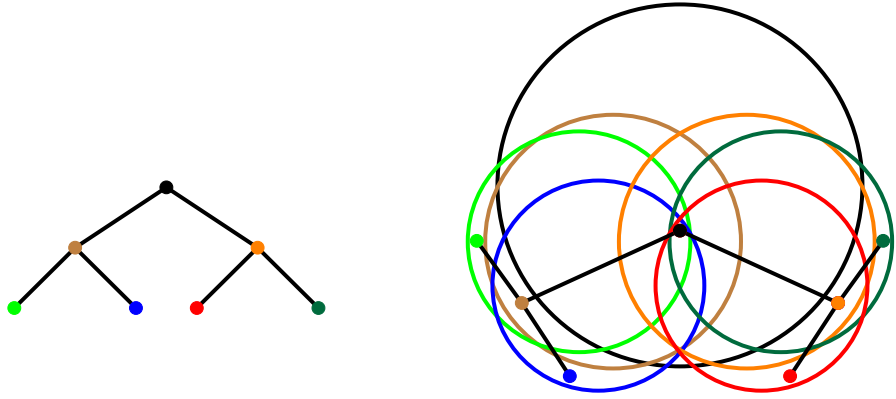

     We prove the existence of the required point set by describing an algorithm that produces a solution.  Let $\Pc=\{p_1,p_2,\dots, p_n\}$ denote the vertices of $T$ in a siblings first order. We will create the planar realization gradually, step by step. The algorithm starts by setting $p_i=q_i$ and in each step we will update the position of some of the $p_i$-s. That is, we identify the vertices of the hypergraph with planar points, and we will update the position of the vertices until they realize the hypergraph $\Hc'(T)$.
     
     During the algorithm, we need to change the position of the points many times without altering the hyperedges that we have already realized. For this reason, we introduce a set of fixed points, $\Pfix$, and a set of disks, $\Fix$, that corresponds to the descendent hyperedges. Once a point is in $\Pfix$  we will not change its position any more. Every disk that we create will be immediately added to $\Fix$, and we will never change its position. The unfixed points, i.e., the points in $\Pc\setminus\Pfix$, will always be kept on the boundary of $\bigcup\limits_{D\in \Fix} D$. Furthermore, if we add $B(v)$ to $\Fix$ in the $k$-th step of the algorithm, $B(v)\cap \Pc$ will remain the same after we finished the $k$-th step, i.e., it will contain exactly the points of the path $Q(v)$ in its interior. Moreover, all descendants of $v$ will be on the boundary of $B(v)$ after the $k$-th step, but later they will be moved off.\\
     
     The structure of the algorithm is the following. We go through the vertices in the siblings first order and for each we do the following. By when we arrive at vertex $p_k$, the disk $B(p_k)$ representing the path $Q(p_k)$ from the root to $p_k$, will be already realized. For each child of $p_k$ we will add a new disk, close to $B(p_k)$, that also realizes the same path from the root to $p_k$. Then we will move the points such that each new disk contains exactly one child of $p_k$. We have summarised the structure of the algorithm in the following pseudo code, while the phases of a step are depicted in Figure \ref{fig:kthstep}.
     
     \smallskip
     \begin{algorithm}
     \DontPrintSemicolon
    \SetAlgoLined
    Set $p_i=q_i$, $\Pfix=\emptyset,\Fix=\emptyset$.\\
    Add the disk of $C$, $B(p_1)$, to $\Fix$.\\
      Move $p_1$ inside $B(p_1)$ and add $p_1$ to $\Pfix$.\\
      \For{$k=1$ \KwTo n}{
      \For{each child $r_i$ of $p_k$}{
       Add a disk $B(r_i)$ representing $Q(p_k)$ to $\Fix$.  \tcc*{Using Lemma \ref{lemma:step} }
       Move the required descendants of $p_k$ to the boundary of $B(r_i)$. 
       }
       \For{each child $r_i$ of $p_k$}{
        Move $r_i$ inside $B(r_i)$. \tcc*{now $B(r_i)$ represents $Q(r_i)$ }
        Add $r_i$ to $\Pfix$.\\
       }
      }
     \caption{Structure of the algorithm}
    \end{algorithm}
     \smallskip
     
     To be able to do these steps we also need a parameter $\delta$ that will ensure that the points do not move too much and the disks are close to each other. There will be only three kinds of operations that we do during the algorithm. 
     \begin{enumerate}
     \renewcommand{\labelenumi}{\alph{enumi})}
        \item  We update the position of a vertex by moving it at a distance less than the current value of $\delta$. If it reached its final position, we add it to $\Pfix$.
         \item We add a new disk to $\Fix$ that is $\delta$-close to one of the disks already in $\Fix$.
         \item We decrease the value of $\delta$. 
    \end{enumerate}
   
     We start with  $\delta=\gamma/n^2$. (The reason for this is explained later.) Every time a disk is added or the position of a vertex is changed, we use Observation \ref{obs:pert} for $(\Pfix,\Fix)$ to update $\delta$ to a smaller value, if needed. After any given update of $\delta$, we only move points at distance less than $\delta$. This ensures two things. Firstly, the points do not move far from their original position, and secondly, if we take a new disk that is $\delta$-close to one of the disks, then it contains the same points of $\Pfix$.   
 
     The algorithm makes an initial adjustment on $p_1$, and then there is one step ($k$-th step) for each point $p_k$. After each step the following properties will hold. 
    
     \begin{proplist}
        \item \label{prop:first} Each point $p_i$ has either reached its final position or it lies on the boundary of the disk $B(w)$ where $w$ is the lowest ancestor of $p_i$ for which $B(w)$ is already defined. Also, in the second case $p_i$ does not belong to the closure of any other disk $D'\in \Fix$. 
        \item \label{prop:second} Suppose $p_j$ is the parent of $p_i$ in $T$. Then in the $j$-th step the point $p_i$ is added to $\Pfix$ and the disk $B(p_i)$ is added to $\Fix$.   
        \item \label{prop:third} Each disk in $\Fix$ contains those points of $P$ that correspond to the appropriate hyperedge of $\Hc'(T)$.  
     \end{proplist}
    
    During the initial adjustment we update $p_1$ to lie inside $C$ but $\delta$-close to $q_1$. We add the disk corresponding to the circle $C$ to $\Fix$. We add $p_1$ to $\Pfix$ and then we update $\delta$ by applying Observation \ref{obs:pert} for $(\Pfix,\Fix)$. Clearly properties $\ref{prop:first},\ref{prop:second}$ and $\ref{prop:third}$ are satisfied.  
    
    In the $k$-th step we update the points in $\Des(p_k)$. (If $p_k$ is a leaf, we continue with the next step.) The process is depicted in Figure \ref{fig:kthstep}; we recommend the reader to consult it before reading on. From properties $\ref{prop:first}$ and $\ref{prop:second}$ we know that at the start of the $k$-th step every point in $\Des(p_k)$ lies on the boundary of $B(p_k)$, and they do not belong to any disk in $\Fix$. Suppose $S(p_k)=\{r_1,\dots, r_l\}$.   To maintain the three properties, we want to add the disks $B(r_1),\dots, B(r_l)$, and by the end of the $k$-th step we  want to place the points of $\Des(r_i)$ on the boundary of $B(r_i)$.
    
      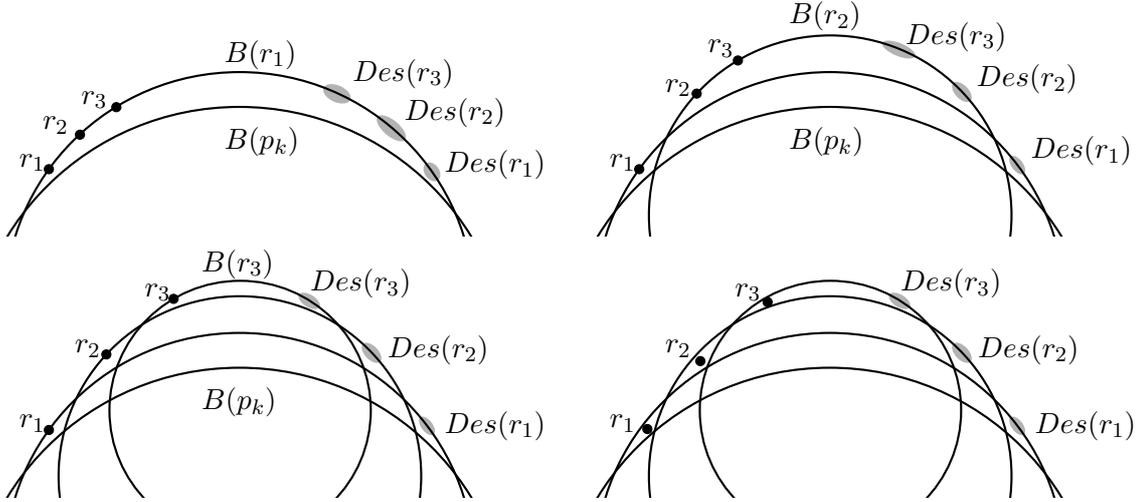
\begin{figure}[!ht]
        \centering
        \scalebox{1.0}{\begin{tikzpicture}[line cap=round,line join=round,>=triangle 45,x=1.0cm,y=1.0cm,scale=1.7]
\clip(4.342139180717218,-5.912130423887219) rectangle (8.76,-4.110066229186957);
\draw [line width=0.8pt] (6.4,-7.) circle (2.09837130335316cm);
\draw [line width=0.8pt] (6.402051166379084,-6.4451889919898715) circle (1.8151700084948583cm);
\fill [rotate around={-40.45359405719779:(7.576154042999593,-5.068234414047189)},line width=0.8pt,fill=black,fill opacity=0.30000001192092896] (7.576154042999593,-5.068234414047189) ellipse (0.138745104653525095cm and 0.06216230811077403cm);
\fill [rotate around={-55.126578882263765:(7.890131372103562,-5.408115776653981)},line width=0.8pt,fill=black,fill opacity=0.30000001192092896] (7.890131372103562,-5.408115776653981) ellipse (0.08058393254481656cm and 0.05715629214874051cm);
\fill [rotate around={-24.603310417430087:(7.155645430710279,-4.799447589705677)},line width=0.8pt,fill=black,fill opacity=0.30000001192092896] (7.155645430710279,-4.799447589705677) ellipse (0.11329648854767155cm and 0.06545837592902208cm);
\draw (4.62,-5.19) node[anchor=north west] {$r_1$};
\draw (4.8,-4.9) node[anchor=north west] {$r_2$};
\draw (5.1,-4.7) node[anchor=north west] {$r_3$};
\draw (7.2,-4.45) node[anchor=north west] {$\Des(r_3)$};
\draw (7.6,-4.75) node[anchor=north west] {$\Des(r_2)$};
\draw (7.911360757629625,-5.161633955392194) node[anchor=north west] {$\Des(r_1)$};
\draw (6.205647383175144,-4.9780161039260514) node[anchor=north west] {$B(p_k)$};
\draw (6.205647383175144,-4.3) node[anchor=north west] {$B(r_1)$};
\begin{scriptsize}
\draw [fill=black] (4.926961488905327,-5.387383431808988) circle (1.01pt);
\draw [fill=black] (5.167411707739564,-5.118318319718645) circle (1.01pt);
\draw [fill=black] (5.448952029243155,-4.903591635818705) circle (1.01pt);
\end{scriptsize}
\end{tikzpicture}
        \begin{tikzpicture}[line cap=round,line join=round,>=triangle 45,x=1.0cm,y=1.0cm,scale=1.7]
\clip(4.342139180717218,-5.912130423887219) rectangle (8.76,-4.0110066229186957);
\draw [line width=0.8pt] (6.4,-7.) circle (2.09837130335316cm);
\draw [line width=0.8pt] (6.402051166379084,-6.4451889919898715) circle (1.8151700084948583cm);
\draw [line width=0.8pt] (6.404633700414337,-5.746650674508501) circle (1.402658952855516cm);
\fill [rotate around={-46.65335223797677:(7.422981265378558,-4.785441918826188)},line width=0.8pt,fill=black,fill opacity=0.30000001192092896] (7.422981265378558,-4.785441918826188) ellipse (0.09644774753534836cm and 0.05295342338148187cm);
\fill [rotate around={-53.06584322136987:(7.851810318717173,-5.355326935741531)},line width=0.8pt,fill=black,fill opacity=0.30000001192092896] (7.851810318717173,-5.355326935741531) ellipse (0.0848580297416653cm and 0.044267673423772476cm);
\fill [rotate around={-22.320519076946074:(6.9351473496602365,-4.454443084531315)},line width=0.8pt,fill=black,fill opacity=0.30000001192092896] (6.9351473496602365,-4.454443084531315) ellipse (0.13711562213590048cm and 0.050928807553377316cm);
\draw (4.62,-5.19) node[anchor=north west] {$r_1$};
\draw (5.05,-4.607488307142961) node[anchor=north west] {$r_2$};
\draw (5.38,-4.3) node[anchor=north west] {$r_3$};
\draw (6.91,-4.16) node[anchor=north west] {$\Des(r_3)$};
\draw (7.46,-4.49) node[anchor=north west] {$\Des(r_2)$};
\draw (7.88,-5.05) node[anchor=north west] {$\Des(r_1)$};
\draw (6.005647383175144,-4.0) node[anchor=north west] {$B(r_2)$};
\draw (6.005647383175144,-4.9780161039260514) node[anchor=north west] {$B(p_k)$};
\begin{scriptsize}
\draw [fill=black] (4.926961488905327,-5.387383431808988) circle (1.01pt);
\draw [fill=black] (5.371509183301354,-4.797911579622135) circle (1.01pt);
\draw [fill=black] (5.690470118571371,-4.539413118600402) circle (1.01pt);
\end{scriptsize}
\end{tikzpicture}}
        \scalebox{1.0}{\begin{tikzpicture}[line cap=round,line join=round,>=triangle 45,x=1.0cm,y=1.0cm,scale=1.7]
\clip(4.342139180717218,-5.912130423887219) rectangle (8.76,-3.9);
\draw [line width=0.8pt] (6.4,-7.) circle (2.09837130335316cm);
\draw [line width=0.8pt] (6.402051166379084,-6.4451889919898715) circle (1.8151700084948583cm);
\draw [line width=0.8pt] (6.404633700414337,-5.746650674508501) circle (1.402658952855516cm);
\draw [line width=0.8pt] (6.406526927693963,-5.234559921602026) circle (1.0109653771645686cm);
\fill [rotate around={-46.65335223797677:(7.422981265378558,-4.785441918826188)},line width=0.8pt,fill=black,fill opacity=0.30000001192092896] (7.422981265378558,-4.785441918826188) ellipse (0.09644774753534836cm and 0.05295342338148187cm);
\fill [rotate around={-53.06584322136987:(7.851810318717173,-5.355326935741531)},line width=0.8pt,fill=black,fill opacity=0.30000001192092896] (7.851810318717173,-5.355326935741531) ellipse (0.0848580297416653cm and 0.044267673423772476cm);
\fill [rotate around={-32.04803973592604:(6.9413070916056565,-4.380327359291241)},line width=1.0pt,fill=black,fill opacity=0.30000001192092896] (6.9413070916056565,-4.380327359291241) ellipse (0.09327448898849997cm and 0.048499800989219445cm);
\draw (4.62,-5.19) node[anchor=north west] {$r_1$};
\draw (5.05,-4.607488307142961) node[anchor=north west] {$r_2$};
\draw (5.59,-4.166789799480186) node[anchor=north west] {$r_3$};
\draw (6.87,-4.05) node[anchor=north west] {$\Des(r_3)$};
\draw (7.470662249966846,-4.568218143093803) node[anchor=north west] {$\Des(r_2)$};
\draw (7.911360757629625,-5.161633955392194) node[anchor=north west] {$\Des(r_1)$};
\draw (6.03,-4.9780161039260514) node[anchor=north west] {$B(p_k)$};
\draw (6.03,-3.9) node[anchor=north west] {$B(r_3)$};
\begin{scriptsize}
\draw [fill=black] (4.926961488905327,-5.387383431808988) circle (1.01pt);
\draw [fill=black] (5.371509183301354,-4.797911579622135) circle (1.01pt);
\draw [fill=black] (5.891794067360919,-4.364444368173137) circle (1.01pt);
\end{scriptsize}
\end{tikzpicture}
        \begin{tikzpicture}[line cap=round,line join=round,>=triangle 45,x=1.0cm,y=1.0cm,scale=1.7]
\clip(4.342139180717218,-5.912130423887219) rectangle (8.76,-4.110066229186957);
\draw [line width=0.8pt] (6.4,-7.) circle (2.09837130335316cm);
\draw [line width=0.8pt] (6.402051166379084,-6.4451889919898715) circle (1.8151700084948583cm);
\draw [line width=0.8pt] (6.404633700414337,-5.746650674508501) circle (1.402658952855516cm);
\draw [line width=0.8pt] (6.406526927693963,-5.234559921602026) circle (1.0109653771645686cm);
\fill [rotate around={-46.65335223797677:(7.422981265378558,-4.785441918826188)},line width=0.8pt,fill=black,fill opacity=0.30000001192092896] (7.422981265378558,-4.785441918826188) ellipse (0.09644774753534836cm and 0.05295342338148187cm);
\fill [rotate around={-53.06584322136987:(7.851810318717173,-5.355326935741531)},line width=0.8pt,fill=black,fill opacity=0.30000001192092896] (7.851810318717173,-5.355326935741531) ellipse (0.0848580297416653cm and 0.044267673423772476cm);
\fill [rotate around={-32.04803973592604:(6.9413070916056565,-4.380327359291241)},line width=1.0pt,fill=black,fill opacity=0.30000001192092896] (6.9413070916056565,-4.380327359291241) ellipse (0.09327448898849997cm and 0.048499800989219445cm);
\draw (4.62,-5.19) node[anchor=north west] {$r_1$};
\draw (5.05,-4.607488307142961) node[anchor=north west] {$r_2$};
\draw (5.59,-4.166789799480186) node[anchor=north west] {$r_3$};
\draw (6.87,-4.05) node[anchor=north west] {$\Des(r_3)$};
\draw (7.470662249966846,-4.568218143093803) node[anchor=north west] {$\Des(r_2)$};
\draw (7.911360757629625,-5.161633955392194) node[anchor=north west] {$\Des(r_1)$};
\begin{scriptsize}
\draw [fill=black] (4.99,-5.38) circle (1.01pt);
\draw [fill=black] (5.4,-4.85) circle (1.01pt);
\draw [fill=black] (5.92,-4.39) circle (1.01pt);
\end{scriptsize}
\end{tikzpicture}}
        \caption{Phases of the $k$-th step of the algorithm.}\label{fig:kthstep}
    \end{figure}
    
    To achieve this, we apply Observation \ref{obs:pert} and Lemma \ref{lemma:step} for each child in the following way. First apply Observation \ref{obs:pert} for the points in $\Des(p_k)$ and disks in $\Fix \setminus \{B(p_k)\}$. Since the points in $\Des(p_k)$ do not belong to the boundary of any other disk this is possible. We update $\delta$ to the value of $\eps$ obtained from Observation \ref{obs:pert} if $\eps<\delta$. 
    
    Then we apply Lemma \ref{lemma:step} for the boundary of the disk $B(p_k)$, such that the $b_i$-s are the points $\{r_1,\dots, r_l\}\cup \Des(r_l)\cup \dots \cup \Des(r_1)$ and $\eps$ is chosen to be the current value of $\delta$. The points $a$ and $c$ have to be chosen carefully. We know that the points in $\{r_1,\dots, r_l\}\cup \Des(r_l)\cup \dots \cup \Des(r_1)$ lie on an arc of $B(p_k)$ that is not covered by any disk in $\Fix$. We choose $a$ and $c$ on the two ends of this arc such that they are also not covered by any disk. Lemma \ref{lemma:step} gives us a circle $C'$, this defines $B(r_1)$, which is added to $\Fix$. The position of the points in $\{r_1,\dots, r_l\}\cup \Des(r_l)\cup \dots \cup \Des(r_1)$ is updated according to the result of Lemma \ref{lemma:step}. As usual, $\delta$ is also updated. 
    
    
    When we apply Observation \ref{obs:pert} for the $i$-th time ($i>1$) we apply it for the points $\{r_i,\dots, r_l\}\cup \Des(r_l)\cup \dots \cup \Des(r_i)$ and disks in $\Fix \setminus \{B(r_{i-1})\}$.
    
    When we apply Lemma \ref{lemma:step} for the $i$-th time ($i>1$), we apply it for the boundary of $B(r_{i-1})$, such that the $b_i$-s are the points in $\{r_i,\dots, r_l\}\cup \Des(r_l)\cup \dots \cup \Des(r_i)$ and $\eps$ is the current value of $\delta$. The point $a$ is chosen between $r_{i-1}$ and $r_i$. The point $c$ is chosen after the points of $\{r_i,\dots, r_l\}\cup \Des(r_l)\cup \dots \cup \Des(r_i)$ but before the points of $Des(r_{i-1})\cup\dots\cup D(r_{1})$. If $Des(r_{i-1})\cup\dots\cup D(r_{1})$ is empty, $c$ is chosen before the arc of $B(r_{i-1})$ reaches any other disk.    
    In the $i$-th case we get $B(r_i)$, which is added to $\Fix$, and new positions for the points $\{r_i,\dots, r_l\}\cup \Des(r_l)\cup \dots \cup \Des(r_i)$ on $B(r_i)$. We update $\delta$ after each application of Lemma \ref{lemma:step}.
    
    Finally, we finish the $k$-th step by moving $r_1,r_2, \dots, r_l$ inside $B(r_1),B(r_2),\dots, B(r_l)$ respectively, but at most $\delta$ far. Since $r_i$ was lying on the boundary of $B(r_1)$, we can also ensure that they are not moved into any other disk. Then we add $r_1,\dots, r_l$ to $\Pfix$. We update $\delta$ again by applying Observation \ref{obs:pert} to $(\Pfix,\Fix)$.
  
    Let us see why properties $\ref{prop:first},\ref{prop:second},\ref{prop:third}$ remain true during the run of the algorithm. Suppose they are true after the $(k-1)$-th step.

    The first part of property $\ref{prop:first}$ and property $\ref{prop:second}$ are maintained since we have created the disks $B(r_1),\dots, B(r_l)$, and by the end of the $k$-th step the points of $\Des(r_i)$ are on the boundary of $B(r_i)$. Points $r_1,\dots, r_l$ were added to $\Pfix$.
    
    The second part of property $\ref{prop:first}$ could be violated in two ways. It could be that one of the new disks covers a point it should not. We have always chosen $a$ and $c$ such that this is avoided. The other possible violation is that we move a point into a disk. This is avoided, since if a point lies on the boundary of disk $D$ before moving it, then we have updated $\delta$ for the disks in $\Fix\setminus\{D\}$ right before moving the point. Also the movement is done by Lemma \ref{lemma:step} so $v$ cannot move into $D$.

    As for property $\ref{prop:third}$, note that the only new disks in $\Fix$ are $B(r_1),\dots, B(r_l)$. Since $\ref{prop:third}$ was true before the step, $B(p_k)$ contains the vertices of $Q(p_k)$ and each of those are in $\Pfix$. When we add $B(r_1)$, it is $\delta$-close to $B(p_k)$, so it contains exactly the points in $Q(p_k)$. Similarly, when $B(r_i)$ is added, it is $\delta$-close to $B(r_{i-1})$, so each of $B(r_1),\dots, B(r_l)$ contains the vertices of $Q(p_k)$. When finally we move the points $r_1,r_2, \dots, r_l$ inside $B(r_1),B(r_2),\dots, B(r_l)$, respectively, we achieve that $B(r_{i})$ contains the vertices of $Q(r_i)$.
    
    We also need to check property $\ref{prop:third}$ for the disks that were already in $\Fix$. Consider a disk $D\in \Fix$. No point inside $D$ was moved, since they are in $\Pfix$. The points in $\Pc\setminus \Pfix$ remain outside of $D$, since $r_i$ only moves into one of the new disks and the rest of the points remain on the boundary of $\bigcup\limits_{D\in \Fix} D$. Hence property $\ref{prop:third}$ remains true.

    
    Finally, we need to show that we get a solution for Theorem \ref{thm:main}. Let $\Dc$ contain the disks in $\Sib$ and those disks in $\Fix$ that correspond to descendent edges that end at a leaf. Property $\ref{prop:third}$ and the argument for the sibling edges tells us that $(\Pc,\Dc)$ is a planar representation of $\Hc(T)$. 
    
    For property $(I)$, note that each point moves less than $n^2$ times since in the $k$-th step they move less than  $n$ times. We started with $\delta=\gamma/{n^2}$, so each point is at most $n^2\frac{\gamma}{n^2}=\gamma$ far from their original position. The first disk was given by $C$ and each disk was taken $\delta$-close to a previous disk, in at most $n$ steps reaching back to the first disk. Hence, all disks are $\gamma/n$ close to $C$, implying $(II)$. 
\end{proof}

\section{A point set that is not three-colorable}\label{sec:3}

Here we prove Theorem \ref{thm:main} by realizing for any $m$ a non-three-colorable $m$-uniform hypergraph with disks.

We introduce a general operation for constructing hypergraphs that are not $c$-colorable, which is similar to prolonging $\Hc(T)$ with the children of a leaf. Essentially the same construction was used in  \cite{MR4012917}.

\begin{defi}\label{def:extension}
Suppose we have a hypergraph $\Ac$, a hypergraph $\Bcal$ and let $F$ be an edge of $\Ac$. Then we define \emph{``$\Ac$ extended by $\Bcal$ through $F$''} as follows. Take $|V(\Bcal)|$ copies of the edge $F$ and add a new vertex to each copy. Then we have $|V(\Bcal)|$ new vertices. Add a set of edges such that they form the hypergraph $\Bcal$ on these new vertices. The resulting hypergraph is denoted by $\Ac+_F\Bcal$. 
\end{defi}

Suppose $\Ac$ is $m$-uniform and $\Bcal$ is $(m+1)$-uniform. Then, if we extend $\Ac$ by $\Bcal$ through each edge of $\Ac$, we get a hypergraph that is $(m+1)$-uniform.

\begin{claim}\label{clam:color}
Suppose $\Ac$ is not $c$-colorable and $\Bcal$ is not $(c-1)$-colorable. Then any extension of $\Ac$ by $\Bcal$ is not $c$-colorable.
\end{claim}
\begin{proof}
 Since any coloring of $\Ac$ using $c$-colors has a monochromatic edge, and we have only ``lost'' the edge $F$, our only chance is to color the extended hypergraph such that the vertices of $F$ are monochromatic. These are contained in each new edge, implying that none of the new vertices have the same color as the vertices of $F$. But then the copy of $\Bc$ is $(c-1)$-colored, and thus one of its edges is monochromatic.  
\end{proof} 

Let $\Gc_i$ denote the hypergraph that has $i$ vertices and only one edge that contains all the vertices. Clearly, $\Gc_1$ is not $c$-colorable for any $c$ and $\Gc_i$ is not 1-colorable. Hence, we can build non-$c$-colorable hypergraphs starting from these trivial ones and using them in the extensions.    

\begin{obs}
    For any rooted tree $T$ the hypergraph $\Hc(T)$ can be built with a series of extensions starting from $\Gc_1$, where each extending hypergraph is one of the $\Gc_i$-s.
\end{obs}

Now we are ready to define the non-three-colorable hypergraphs that we will realize with points and disks.  For each $m$ let $\Hc^2(m)=\Hc(T)$ where $T$ is the $m$-ary tree of depth $m$. Clearly, $\Hc^2(m)$ is $m$-uniform and not two-colorable. We define non-three-colorable $m$-uniform hypergraphs $\Hc^3(m)$ based on them inductively. 

First we create a sequence of hypergraphs. Let $\Fcal_1(m)=\Gc_1$ and let $v$ denote the single vertex of it. For for $i>1$, let $\Fcal_i(m)$ be the hypergraph that we get if we extend each edge of $\Fcal_{i-1}(m)$ that contains $v$ by $\Hc^2(m)$. Note that $\Fcal_i(m)$ has only two types of edges. The first type are the ones that contain $v$, each of these contains exactly $i$ vertices. (They are like the descendent edges.) The second type are those that were added in a copy of $\Hc^2(m)$, these contain exactly $m$ vertices. (They are a bit like the sibling edges.) 
Therefore, $\Fcal_m(m)$ is $m$-uniform. Also, by Claim \ref{clam:color} and induction, no $\Fcal_i(m)$ is three-colorable. 

Let $\Hc^3(m)=\Fcal_m(m)$. For example, $\Fcal_1(2)$ is $\Gc_1$ and $\Hc^2(2)$ is a triangle graph. Extending $\Gc_1$ through its single edge by a triangle gives us the complete graph on $4$ vertices, so $\Hc^3(2)$ is just $K_4$.  
 
 \subsection{Realizing $\Hc^3(m)$}
 
 Since $\Hc^3(m)$ was built by a series of extensions, it is enough to show that we can do each extension geometrically. 
 This step is essentially the same as in \cite{MR2364757}.
 
 \begin{lemma}\label{lemma:extension}
  Suppose a hypergraph $\Ac$ is realized with $(\Pc,\Dc)$ such that every disk $D\in \Bc \subset \Dc$ has a point on its boundary that does not belong to the closure of any other disk $D'\in \Dc$. Then we can also realize $\Ac+_F \Hc(T)$ for any rooted tree $T$ and any edge $F\in \Bc$, with a pair $(\Pc',\Dc')$, such that any disk $D\in \Bc$ and every new copy of $F$ has a point on its boundary that does not belong to the closure of any other disk $D'\in \Dc'$.  
 \end{lemma}
 \begin{proof}
 
 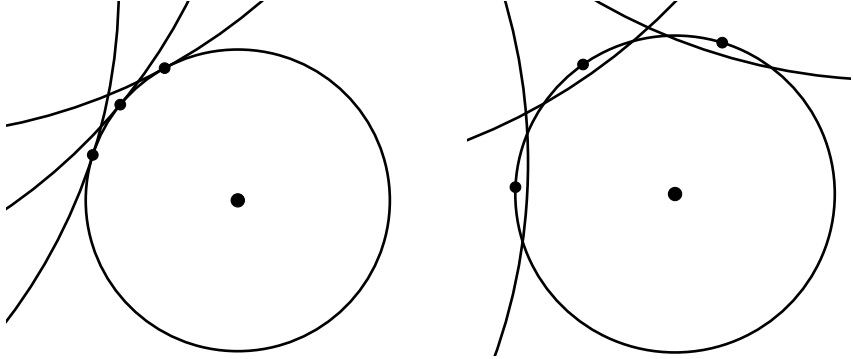
\begin{figure}[!ht]
     \centering
     \scalebox{1.0}{
    \definecolor{uuuuuu}{rgb}{0.,0.,0.}
    \definecolor{xdxdff}{rgb}{0.,0.,0.}
    \definecolor{ududff}{rgb}{0.,0.,0.}
    \begin{tikzpicture}[line cap=round,line join=round,>=triangle 45,x=1.0cm,y=1.0cm]
    \clip(-1.044117164263055,-0.051409109231156304) rectangle (4.9,4.635371044692869);
    \draw [line width=1.pt] (2.,2.) circle (2.cm);
    \draw [line width=1.pt] (2.,2.) circle (9.362666077956689cm);
    \draw [line width=1.pt] (-6.927584110616714,4.820949846272505) circle (7.36266607795669cm);
    \draw [line width=1.pt] (-5.2356168707059885,7.9418317535652045) circle (7.362666077956687cm);
    \draw [line width=1.pt] (-2.5033975073380645,10.20846678608266) circle (7.362666077956691cm);
    \begin{scriptsize}
    \draw [fill=ududff] (2.,2.) circle (2.5pt);
    \draw [fill=ududff] (-6.927584110616714,4.820949846272505) circle (2.5pt);
    \draw [fill=xdxdff] (-5.2356168707059885,7.9418317535652045) circle (2.5pt);
    \draw [fill=ududff] (-2.5033975073380645,10.20846678608266) circle (2.5pt);
    \draw [fill=uuuuuu] (0.09293975961918056,2.6025954194637153) circle (2.0pt);
    \draw [fill=uuuuuu] (0.4543682727847338,3.2692606366800927) circle (2.0pt);
    \draw [fill=uuuuuu] (1.0380095861923797,3.7534464473551052) circle (2.0pt);
    \end{scriptsize}
    \end{tikzpicture}
    \begin{tikzpicture}[line cap=round,line join=round,>=triangle 45,x=1.0cm,y=1.0cm]
    \clip(-1.044117164263055,-0.051409109231156304) rectangle (3.981058668644672,4.635371044692869);
    \draw [line width=1.pt] (1.6800281588004644,2.084203116105142) circle (9.664648100444605cm);
    \draw [line width=1.pt] (-7.975613653929894,2.5013410904981606) circle (7.7223145349873565cm);
    \draw [line width=1.pt] (-3.886638773837614,9.984686782723362) circle (7.722314534987356cm);
    \draw [line width=1.pt] (4.536228073061077,11.317162823328237) circle (7.722314534987356cm);
    \draw [line width=1.pt] (1.6800281588004644,2.084203116105142) circle (2.10083322008578cm);
    \begin{scriptsize}
    \draw [fill=ududff] (1.6800281588004644,2.084203116105142) circle (2.5pt);
    \draw [fill=ududff] (-7.975613653929894,2.5013410904981606) circle (2.5pt);
    \draw [fill=xdxdff] (-3.886638773837614,9.984686782723362) circle (2.5pt);
    \draw [fill=ududff] (4.536228073061077,11.317162823328237) circle (2.5pt);
    \draw [fill=uuuuuu] (-0.41884733781488,2.1748776347069967) circle (2.0pt);
    \draw [fill=uuuuuu] (0.469985258562732,3.801554681219357) circle (2.0pt);
    \draw [fill=uuuuuu] (2.300888804808757,4.091198949896073) circle (2.0pt);
    \end{scriptsize}
\end{tikzpicture}}
    \caption{Realizing an extension.}
    \label{fig:extension}
 \end{figure}

Suppose  $D_F\in \Dc$ is the disk realizing the edge $F$. Then $D_F$ has a point $p$ on its boundary that does not belong to the closure of any other disk $D'\in \Dc$. Take a small circle $C$ that is tangent to $D$ at $p$. If we chose the radius of $C$ to be small enough, then $C$ will not intersect any of the disks. Now take $n=|V(T)|$ copies of $D_F$ and rotate them slightly around the center of $C$. If all the rotations are small, the resulting disks will be $\eps$-close to $D_F$ and by Observation \ref{obs:pert} they will contain the same points as $D_F$. Also, if the angles of the rotations are different, then each new disk has a point on its boundary that does not belong to the closure of any other disk. 

Then we enlarge each copy of $D_F$ slightly, so that they intersect $C$ but some part of their boundary remains uncovered by other disks. Place the points $q_1,\dots, q_{n}$ in the intersections and use Theorem $\ref{thm:construction}$ to realize $\Hc(T)$. Since each disk in the realization of $\Hc(T)$ is close to $C$, they do not contain any point of $\Pc$. 
\end{proof}

 \section{Stabbed unit disks}\label{sec:4}

Here we prove Theorem \ref{thm:unit}, that given a fixed origin $o$, any point set $\Pc$ can be $k$-colored such that every unit disk that contains the origin and at least $8k-7$ points of $\Pc$, contains all $k$ colors.

We call unit disks containing $o$ \emph{stabbed} unit disks.
Take a representative disk for every $\Pc'\subset \Pc$ that can be obtained as the intersection of a stabbed unit disk and $\Pc$.
Since $\Pc$ is finite, the collection of these representative disks, $\Dc$, is also finite.
We divide the plane into four quarters by two perpendicular lines through $o$.


\begin{lemma}\label{lemma:disk2pseudo}
The boundaries of two stabbed unit disks intersect at most once in each quarter.
\end{lemma}

(See Figure \ref{fig:disk2pseudo}.)

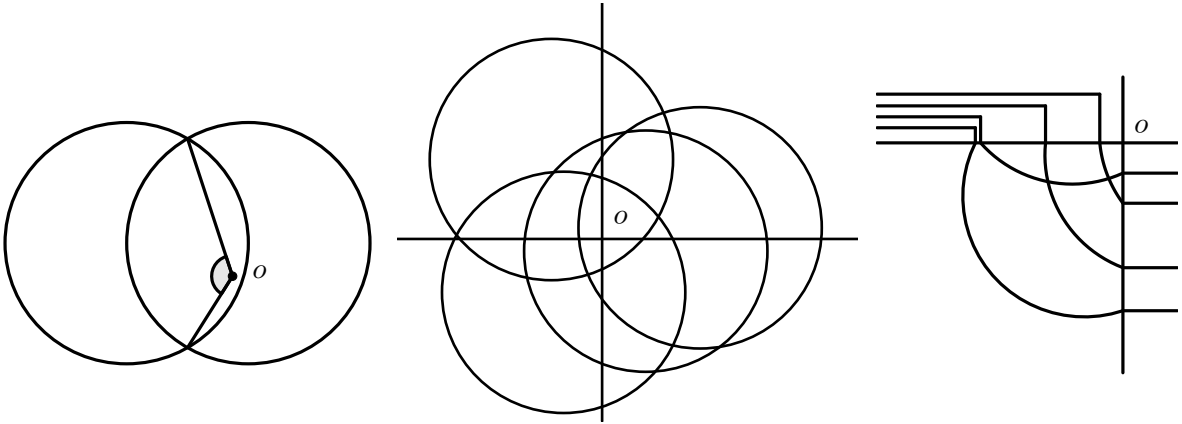
\begin{figure}[!ht]
    \centering
    \begin{tikzpicture}[line cap=round,line join=round,>=triangle 45,x=1.0cm,y=1.0cm,scale=1.6]
\clip(-3.042386084419488,0.5245277774578191) rectangle (0.06702136573879124,3.437550519830568);
\draw [shift={(-1.129785553382904,1.7262166228024929)},line width=1.2pt,fill=black,fill opacity=0.10000000149011612] (0,0) -- (107.99400368162406:0.17179046685957342) arc (107.99400368162406:237.99022165733754:0.17179046685957342) -- cycle;
\draw [line width=1.2pt] (-2.,2.) circle (1.cm);
\draw [line width=1.2pt] (-1.,2.) circle (1.cm);
\draw [line width=1.2pt] (-1.5,2.8660254037844384)-- (-1.129785553382904,1.7262166228024929);
\draw [line width=1.2pt] (-1.129785553382904,1.7262166228024929)-- (-1.5,1.1339745962155614);
\draw (-1.05,1.9) node[anchor=north west] {$o$};

\begin{scriptsize}
\draw [fill=black] (-1.129785553382904,1.7262166228024929) circle (1.0pt);
\end{scriptsize}
\end{tikzpicture}
    \begin{tikzpicture}[rotate=180,transform shape,line cap=round,line join=round,>=triangle 45,x=1.6cm,y=1.6cm,scale=1.0]
\clip(-0.1,0.04556402153881431) rectangle (3.6831530352809585,3.5102743023700134);
\draw [line width=1.pt] (2.316184676178499,2.442482615310425) circle (1.6cm);
\draw [line width=1.pt] (2.416942502776672,1.341943119043924) circle (1.6cm);
\draw [line width=1.pt] (1.194319109692564,1.9079677066168883) circle (1.6cm);
\draw [line width=1.pt] (2.,0.04556402153881431) -- (2.,3.5102743023700134);
\draw [line width=1.pt,domain=-0.2442040189550003:3.6831530352809585] plot(\x,{(--8.-0.*\x)/4.});
\draw [line width=1.pt] (1.6411533875648765,2.1011220617338386) circle (1.6cm);
\draw (1.7,1.96) node[anchor=north west] {$o$};
\begin{scriptsize}
\end{scriptsize}
\end{tikzpicture}
    \begin{tikzpicture}[rotate=180,transform shape,line cap=round,line join=round,>=triangle 45,x=1.6cm,y=1.6cm,scale=1.0]
\clip(1.5509114738689072,1.0541009213295384) rectangle (4.021862542863723,4.305646615436978);
\draw [line width=1.2pt] (2.,1.4541009213295384) -- (2.,3.905646615436978);
\draw [line width=1.2pt,domain=1.5509114738689072:4.021862542863723] plot(\x,{(--8.-0.*\x)/4.});
\draw (1.7,2.0) node[anchor=north west] {$o$};
\draw [shift={(2.316184676178499,2.442482615310425)},line width=1.2pt]  plot[domain=-0.4583651638222923:1.892501460857638,variable=\t]({1.*1.*cos(\t r)+0.*1.*sin(\t r)},{0.*1.*cos(\t r)+1.*1.*sin(\t r)});
\draw [shift={(2.416942502776672,1.341943119043924)},line width=1.2pt]  plot[domain=0.7182352290576206:2.000875209464068,variable=\t]({1.*1.*cos(\t r)+0.*1.*sin(\t r)},{0.*1.*cos(\t r)+1.*1.*sin(\t r)});
\draw [shift={(1.194319109692564,1.9079677066168883)},line width=1.2pt]  plot[domain=0.09216270912031764:0.6339722821812452,variable=\t]({1.*1.*cos(\t r)+0.*1.*sin(\t r)},{0.*1.*cos(\t r)+1.*1.*sin(\t r)});
\draw [shift={(1.6411533875648765,2.1011220617338386)},line width=1.2pt]  plot[domain=-0.1012951997887912:1.2037644160306933,variable=\t]({1.*1.*cos(\t r)+0.*1.*sin(\t r)},{0.*1.*cos(\t r)+1.*1.*sin(\t r)});
\draw [line width=1.2pt,domain=1.5509114738689072:2.0] plot(\x,{(-0.9745474471442748-0.*\x)/-0.2873770687361441});
\draw [line width=1.2pt,domain=1.5509114738689072:2.0] plot(\x,{(-0.7935364338000417-0.*\x)/-0.2615032355529523});
\draw [line width=1.2pt,domain=1.5509114738689072:2.0] plot(\x,{(-0.7508803194160878-0.*\x)/-0.3003139853277399});
\draw [line width=1.2pt,domain=1.5509114738689072:2.0] plot(\x,{(-0.47213377277472457-0.*\x)/-0.20975556918656868});
\draw [line width=1.2pt,domain=2.190075132508718:4.021862542863723] plot(\x,{(--0.3633578038262244-0.*\x)/0.22760975299711417});
\draw [line width=1.2pt,domain=2.636027414078132:4.021862542863723] plot(\x,{(--0.5065638785122175-0.*\x)/0.2991341350915362});
\draw [line width=1.2pt,domain=3.212961758372311:4.021862542863723] plot(\x,{(--0.6432929175273499-0.*\x)/0.34317178719396013});
\draw [line width=1.2pt] (2.190075132508718,2.)-- (2.190075132508718,1.5964070038370943);
\draw [line width=1.2pt] (2.636027414078132,2.)-- (2.636027414078132,1.693433878274064);
\draw [line width=1.2pt,domain=3.1699107235047186:4.021862542863723] plot(\x,{(--0.6197804599464174-0.*\x)/0.34741207228676485});
\draw [line width=1.2pt] (3.1699107235047186,2.)-- (3.1699107235047186,1.7839922944152362);
\draw [line width=1.2pt] (3.212961758372311,1.874550710556408)-- (3.212961758372311,2.);
\end{tikzpicture}
    \caption{From stabbed disks to pseudolines.}
    \label{fig:disk2pseudo}
\end{figure}

\begin{proof}
Denote the two intersection points of the boundary of the two disks, $D_1$ and $D_2$, by $x_1$ and $x_2$.
Since $D_1$ and $D_2$ have equal radii, the length of the arc they contain from each other is less than half of their perimeter.
Thus, using Thales's theorem, for any inner point $p$ of $D_1\cap D_2$ the angle $x_1px_2$ is at least $90^\circ$.
This proves the statement as $o\in D_1\cap D_2$.
\end{proof}

This means that in each quarter the boundaries of the disks behave as pseudolines.
Let us partition our point set $\Pc$ into four parts, $\Pc_1,\Pc_2,\Pc_3,\Pc_4$, depending on which quarter the points are in.
We will color each $\Pc_i$ separately.

Denote the quarter that contains the points of $\Pc_i$ by $Q_i$.
Prolong the section of the boundaries falling into this quarter, i.e., $D\cap Q_i$ for each $D\in \Dc$, such that they do not intersect outside $Q_i$ and all of them stretches from ``left infinity'' to ``right infinity''.
Without loss of generality, we can suppose that $o$ is above all these curves.
The (upward) regions that are bounded by such curves form (upward) \emph{pseudohalfplanes}.
Generalizing a theorem of Smorodinsky and Yuditsky \cite{MR2844088}, the following polychromatic coloring result is known about them.

\begin{theorem}[Keszegh-P\'alv\"olgyi \cite{abafree}]
 Given a finite collection of points and pseudohalfplanes, the points can be $k$-colored such that every pseudohalfplane that contains at least $2k-1$ points will contain all $k$ colors.
\end{theorem}

We can apply this theorem to $\Pc_i$ (ignoring $\Pc\setminus \Pc_i$) and our regions.
From this we get that any stabbed disk that contains at least $2k-1$ points of $\Pc$ from the quarter $Q_i$ will contain all $k$ colors.
Repeating this for all four $\Pc_i$, we obtain by the pigeonhole principle that any stabbed disk containing at least $4(2k-2)+1=8k-7$ points of $\Pc$ contains all $k$ colors.

For $k=2$, this gives that 9 points per disk are enough, though this is probably far from being tight.
From below, we could only find a construction showing that 3 points per disk are not enough.

\section{Generalizations to other shapes and connection to clustering}\label{sec:5}

\subsubsection*{Theorem \ref{thm:gen_main}: Generalizing Theorem \ref{thm:main}}

Let $(\Pc,\Sc)$ be a range space where $\Sc$ contains homothets of a given convex sets $C$. Theorem~\ref{thm:main} shows that if $C$ is a disk, then the hypergraph $\Hc(\Pc,\Sc)$ might not be three-colorable. For which other $C$ convex sets can we generalize this statement?  Theorem \ref{thm:gen_main} shows that the answer to this question is a surprisingly wide family. We omit the detailed proof of Theorem \ref{thm:gen_main}, the main ideas are the following. 

It is not hard to see, that to achieve an appropriate version of Theorem \ref{lemma:step}, it is enough to assume that $C$ is closed and strictly convex. Furthermore, since we can restrict ourselves to an arbitrarily small portion of the boundary of $C$, the non-two-colorable construction of Theorem \ref{thm:construction} can be achieved by any closed $C$  that is strictly convex in the neighbourhood of a point on the boundary. 

The geometric extension in Lemma \ref{lemma:extension} requires a bit more. Let $C$ be any closed  convex set in the plane, that has two parallel supporting lines such that $C$ is strictly convex in some neighbourhood of the two points of tangencies. When we extend through an edge, we can use one of these neighbourhoods to correctly place the new points and the other neighbourhood to create the extending hypergraph on the new points.

\subsubsection*{Theorems \ref{thm:convex} and \ref{thm:polygon}: Generalizing Theorem \ref{thm:unit}}

The proofs of Theorems \ref{thm:convex} and \ref{thm:polygon} follow the same steps as of Theorem \ref{thm:unit}. We divide the plane into regions using rays starting at $o$ and then we show that in each region the boundaries of the sets behave as pseudolines. Here we provide the sketch of the proof of Theorem \ref{thm:convex}, but leave the details to the interested reader.

\begin{proof}[Sketch of the proof of Theorem \ref{thm:convex}]
The proof of Theorem \ref{thm:unit} relied on the fact that if $p$ is a point in the intersection of two unit disks, then it sees the two intersection points $x_1,x_2$ in an angle that is at least $\frac{\pi}{2}$. For a general convex set $C$ this angle might be arbitrarily small if the set is thin enough. 

To avoid this problem, we can apply an affine transformation to $C$ to make its width more even from different directions. 
An affine transformation does not change the conclusion of Threom \ref{thm:convex}, so it is indeed enough to argue for the affine copy.
One way to achieve this is via the John ellipsoid of the set, i.e., the largest volume ellipsoid contained in the set. Take the affine transformation that changes the John ellipsoid to be the unit circle. It is well known that in this case the set is contained in the circle whose radius is 2 and shares the center with the John ellipsoid.

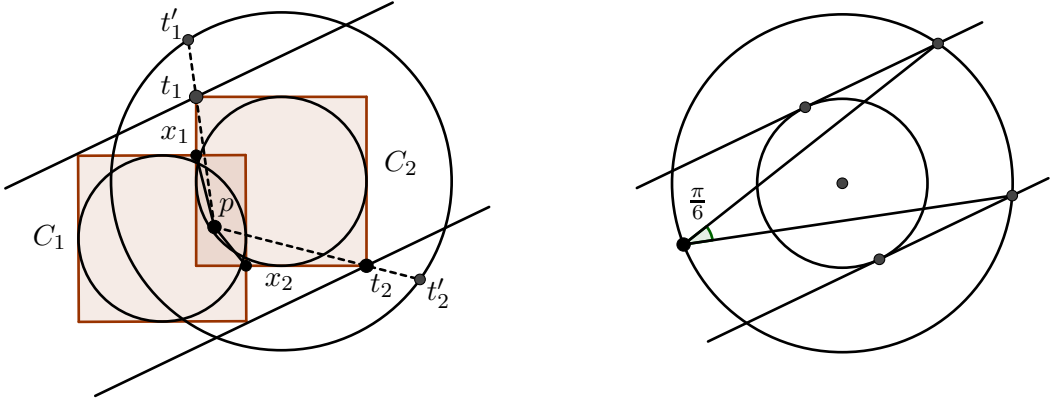
\begin{figure}
    \centering
    \definecolor{qqwuqq}{rgb}{0.,0.39215686274509803,0.}
\definecolor{uuuuuu}{rgb}{0.26666666666666666,0.26666666666666666,0.26666666666666666}
\definecolor{zzttqq}{rgb}{0.6,0.2,0.}
        \scalebox{1.0}{
       \begin{tikzpicture}[line cap=round,line join=round,>=triangle 45,x=1.4cm,y=1.4cm]
\clip(-9.281040768468834,5.474626195181661) rectangle (0.5511986473569976,9.403849789928175);
\fill[line width=1.pt,color=zzttqq,fill=zzttqq,fill opacity=0.10000000149011612] (-8.572840106661644,6.287128046219646) -- (-6.9993664458295335,6.293671632384209) -- (-7.005910031994096,7.8671452932163195) -- (-8.579383692826207,7.860601707051757) -- cycle;
\fill[line width=1.pt,color=zzttqq,fill=zzttqq,fill opacity=0.10000000149011612] (-7.470020993120913,6.818525508007105) -- (-5.870020993120912,6.818525508007105) -- (-5.870020993120912,8.418525508007106) -- (-7.470020993120913,8.418525508007107) -- cycle;
\draw [shift={(-2.8910910284291473,7.0197866384353835)},line width=1.pt,color=qqwuqq,fill=qqwuqq,fill opacity=0.10000000149011612] (0,0) -- (8.536644550474122:0.27541286879063953) arc (8.536644550474122:38.536644550473646:0.2754128687906395) -- cycle;
\draw [line width=1.pt,color=zzttqq] (-8.572840106661644,6.287128046219646)-- (-6.9993664458295335,6.293671632384209);
\draw [line width=1.pt,color=zzttqq] (-6.9993664458295335,6.293671632384209)-- (-7.005910031994096,7.8671452932163195);
\draw [line width=1.pt,color=zzttqq] (-7.005910031994096,7.8671452932163195)-- (-8.579383692826207,7.860601707051757);
\draw [line width=1.pt,color=zzttqq] (-8.579383692826207,7.860601707051757)-- (-8.572840106661644,6.287128046219646);
\draw [line width=1.pt,color=zzttqq] (-7.470020993120913,6.818525508007105)-- (-5.870020993120912,6.818525508007105);
\draw [line width=1.pt,color=zzttqq] (-5.870020993120912,6.818525508007105)-- (-5.870020993120912,8.418525508007106);
\draw [line width=1.pt,color=zzttqq] (-5.870020993120912,8.418525508007106)-- (-7.470020993120913,8.418525508007107);
\draw [line width=1.pt,color=zzttqq] (-7.470020993120913,8.418525508007107)-- (-7.470020993120913,6.818525508007105);
\draw [line width=1.pt] (-7.78937506932787,7.077136669717983) circle (1.101441087cm);
\draw [line width=1.pt] (-6.670020993120913,7.618525508007106) circle (1.12cm);
\draw [line width=1.pt] (-6.670020993120913,7.618525508007106) circle (2.24cm);
\draw [line width=1.pt] (-7.298188442770469,7.186097411780155)-- (-7.470020993120914,7.865215200492397);
\draw [line width=1.pt] (-7.298188442770469,7.186097411780155)-- (-7.001549149388184,6.818525508007104);
\draw [line width=1.pt,dash pattern=on 2pt off 2pt] (-7.298188442770469,7.186097411780155)-- (-7.545228411797118,8.957932830032092);
\draw [line width=1.pt,dash pattern=on 2pt off 2pt] (-7.298188442770469,7.186097411780155)-- (-5.367573040981403,6.6892089139619815);
\draw [line width=1.pt] (-9.262206251129468,7.551713910625877)-- (-5.624018069107316,9.311366656088357);
\draw [line width=1.pt] (-4.718121481176578,7.3756553312267314)-- (-8.418045339743387,5.586143395615268);
\draw [line width=1.pt] (-1.4,7.6) circle (1.12cm);
\draw [line width=1.pt] (-1.4,7.6) circle (2.24cm);
\draw [line width=1.pt] (-2.8910910284291473,7.0197866384353835)-- (-0.5009271786392697,8.9235059735001);
\draw [line width=1.pt] (0.1957262057919096,7.483133083599534)-- (-2.8910910284291473,7.0197866384353835);
\draw (-2.95,7.65) node[anchor=north west] {$\frac{\pi}{6}$};
\draw [line width=1.pt] (-3.3302819727798507,7.555055062687216)-- (-0.08708406823054204,9.123666088642);
\draw [line width=1.pt] (0.535675243661224,7.647553456564294)-- (-2.661469147543755,6.101216757043471);
\draw (-7.35,7.55) node[anchor=north west] {$p$};
\draw (-7.90,8.237935312047783) node[anchor=north west] {$x_1$};
\draw (-6.921670525829022,6.84251011017519) node[anchor=north west] {$x_2$};
\draw (-7.9,8.715317617951566) node[anchor=north west] {$t_1$};
\draw (-5.93018419818272,6.833329681215502) node[anchor=north west] {$t_2$};
\draw (-5.8,8.0) node[anchor=north west] {$C_2$};
\draw (-9.1,7.3) node[anchor=north west] {$C_1$};

\draw (-5.4,6.8) node[anchor=north west] {$t_2'$};
\draw (-7.9,9.35) node[anchor=north west] {$t_1'$};
\begin{scriptsize}
\draw [fill=black] (-5.870020993120912,6.818525508007105) circle (2.5pt);
\draw [fill=uuuuuu] (-7.470020993120913,8.418525508007107) circle (2.5pt);
\draw [fill=black] (-7.298188442770469,7.186097411780155) circle (2.5pt);
\draw [fill=black] (-7.470020993120914,7.865215200492397) circle (2.0pt);
\draw [fill=black] (-7.001549149388184,6.818525508007104) circle (2.0pt);
\draw [fill=uuuuuu] (-7.545228411797118,8.957932830032092) circle (2.0pt);
\draw [fill=uuuuuu] (-5.367573040981403,6.6892089139619815) circle (2.0pt);
\draw [fill=uuuuuu] (-1.4,7.6) circle (2.0pt);
\draw [fill=uuuuuu] (-1.7483266922155898,8.320186444950282) circle (2.0pt);
\draw [fill=uuuuuu] (-1.0516733077844096,6.879813555049715) circle (2.0pt);
\draw [fill=black] (-2.8910910284291473,7.0197866384353835) circle (2.5pt);
\draw [fill=uuuuuu] (-0.5009271786392697,8.9235059735001) circle (2.0pt);
\draw [fill=uuuuuu] (0.1957262057919096,7.483133083599534) circle (2.0pt);
\end{scriptsize}
\end{tikzpicture}
}
    \caption{Angle $x_1px_2$ is at least $\frac{\pi}{6}$.}
    \label{fig:my_label}
\end{figure}

Just like in the proof of Lemma \ref{lemma:disk2pseudo}, we can show that the angle $x_1px_2$ is at least $\frac{\pi}{6}$. Let $C_1$ and $C_2$ be two translates of $C$. $C_1$ and $C_2$ have two common tangent lines, let $t_1,t_2$ be the tangent points to these lines on $C_2$. Let $t'_1,t'_2$ be the projections of $t_1,t_2$ to the circle of radius 2 around $C_2$. Now $x_1px_2$ is bigger than angle $t_1xt_2$. Clearly the distance of $t'_1$ and $t'_2$ is at least 2, since they are separated by the common tangent lines. Hence, $t'_1xt'_2$ is at least as big as the inscribed angle corresponding to a chord of length $2$, which is $\frac{\pi}{6}$. Therefore, we can divide the plane into 12 regions and in each region the boundaries of the sets behave as pseudolines and from here the proof is the same, giving $12(2k-2)+1=12k-23$.     
\end{proof}

We could not rule out the possibility that with a suitable affine transformation the $x_1px_2$ angle would be always at least $\frac\pi2$, just like for disks.
This is a nice geometric question.

\begin{proof}[Proof of Theorem \ref{thm:polygon}]
 Let $P$ be a fixed polygon and let $v_1, v_2, \dots, v_n$ be the vertices of $P$. Let $\alpha_{ij}=\inf \{\angle v_iqv_{j}~|~i\in \{1,\dots,n\}, q\in P\setminus\overline{v_iv_{j}}  \}$. 
 It is not hard to see that this angle $\alpha_{ij}$ is minimized for a vertex $v$.  
 This also implies $\alpha_{ij}>0$.
 Let $\alpha=\min_{i,j} \alpha_{ij}$.
 
 
 Divide the plane into regions around $o$ such that each region is a cone whose angle is less than $\alpha$.
 A homothet of $P$ that contains $o$ can have at most one vertex in each region, otherwise they would be visible from an angle $\alpha$ from $o$.
 Let $P'$ and $P''$ be  homotethic copies of $P$ that contain $o$. 
 We will show that $P'$ and $P''$ intersects at most once in each region. 
 
 Suppose there is a region $R$ where they intersect twice. The boundary of $R$ intersects at most two sides of each polygon, since they are convex and contain $o$. There are essentially two ways $P'$ and $P''$ can intersect (see Figure \ref{fig:polys}). Either the two intersection points of $P'$ and $P''$ fall on the same side of one of the polygons, or not.

\begin{figure}[!h]
    \centering
    \newcommand\mlw{1.0pt}
\begin{tikzpicture}[line cap=round,line join=round,>=triangle 45,x=1.0cm,y=1.0cm]
\clip(-0.14610772782179205,0.8) rectangle (12.508834249795028,6.046957595291608);
\draw [line width=\mlw,domain=-0.24610772782179205:3.0] plot(\x,{(-7.--2.*\x)/-1.});
\draw [line width=\mlw,domain=3.0:12.508834249795028] plot(\x,{(-5.--2.*\x)/1.});
\draw [line width=\mlw,domain=-0.24610772782179205:9.0] plot(\x,{(-19.--2.*\x)/-1.});
\draw [line width=\mlw,domain=9.0:12.508834249795028] plot(\x,{(-17.--2.*\x)/1.});
\draw [line width=\mlw] (1.,4.)-- (4.37524806116981,4.845405090153813);
\draw [line width=\mlw] (5.0,4.)-- (4.37524806116981,4.845405090153813);
\draw [line width=\mlw] (1.619384225742496,2.888746480496106)-- (3.589617577041238,5.330647447997922);
\draw [line width=\mlw] (3.589617577041238,5.330647447997922)-- (4.135911212951387,2.317619788363592);
\draw [line width=\mlw] (7.379129324014351,2.973755995612246)-- (9.089030965941244,3.3203576797866106);
\draw [line width=\mlw] (9.089030965941244,3.3203576797866106)-- (10.590971597363515,2.973755995612246);
\draw [line width=\mlw] (8.187866587087882,1.5642424799698322)-- (9.089030965941244,5.515501679557582);
\draw [line width=\mlw] (9.089030965941244,5.515501679557582)-- (9.8284478921799,1.58734925891479);
\draw (9.204564860666034,1.1252136800156378) node[anchor=north west] {$o$};
\draw (3.2661226718118255,1.1945340168505105) node[anchor=north west] {$o$};
\draw (1.8,4.7) node[anchor=north west] {$s$};
\draw (8.210973366032839,4.7067644164840665) node[anchor=north west] {};
\draw (9.48184620800553,4.799191532263897) node[anchor=north west] {};
\draw (8.0,3.6) node[anchor=north west] {$s$};
\draw (3.3585497875916577,4.036667827080296) node[anchor=north west] {};
\draw (2.503598966628211,4.082881384970212) node[anchor=north west] {};
\draw (9.227671639610993,3.112396669281992) node[anchor=north west] {};
\begin{scriptsize}
\draw [fill=black] (3.,1.) circle (1.5pt);
\draw [fill=black] (5.,4.) circle (1.5pt);
\draw [fill=black] (9.,1.) circle (1.5pt);
\draw [fill=black] (1.,4.) circle (1.5pt);
\draw [fill=black] (4.37524806116981,4.845405090153813) circle (1.5pt);
\draw [fill=black] (1.619384225742496,2.888746480496106) circle (1.5pt);
\draw [fill=black] (3.589617577041238,5.330647447997922) circle (1.5pt);
\draw [fill=black] (4.135911212951387,2.317619788363592) circle (1.5pt);
\draw [fill=black] (7.379129324014351,2.973755995612246) circle (1.5pt);
\draw [fill=black] (9.089030965941244,3.3203576797866106) circle (1.5pt);
\draw [fill=black] (10.590971597363515,2.973755995612246) circle (1.5pt);
\draw [fill=black] (8.187866587087882,1.5642424799698322) circle (1.5pt);
\draw [fill=black] (9.089030965941244,5.515501679557582) circle (1.5pt);
\draw [fill=black] (9.8284478921799,1.58734925891479) circle (1.5pt);
\end{scriptsize}
\end{tikzpicture}
\caption{The two ways $P'$ and $P''$ can intersect each other twice in a region.}
    \label{fig:polys}
\end{figure}
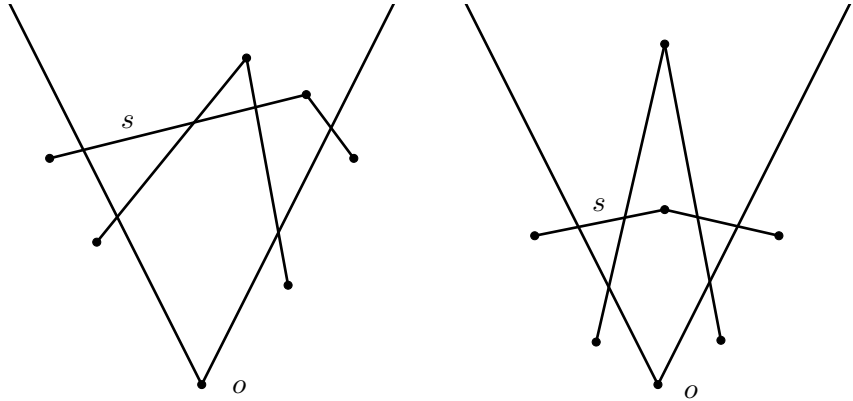

We will show that these are not possible for two homothets. In the first case let $s$ be the side that contains the two intersections. In the second case let $s$ be one of the sides that ends at the vertex in $R$ that is closer to $o$.  Let $n$ be the normal vector of the side $s$ and let us look at the extremal points of $P'$ and $P''$ in the direction of $n$. Since they are homothets of each other, the extremal points should form a side of the polygon in both $P'$ and $P''$. But for one of them we get only a vertex as an extremal point.
Therefore two homothets meet at most once in each region. The rest of the proof follows the same way as for Theorem \ref{thm:unit}.
\end{proof}

\subsubsection*{Connection to clustering for planar graphs}

We say that a graph $G$ is $c$-colorable with clustering $m$ if each vertex can be assigned one of $c$ colors so that each monochromatic component has at most $m$ vertices.
For example, planar graphs are four-colorable with clustering 1 due to the Four Color Theorem. Clustering questions gained popularity recently, see \cite{MR1957474,MR3217360,LINIAL2007115,MR4055209}.
It has been shown \cite{3colcluster} that for each $m$ there exists a planar graph that is not three-colorable with clustering $m$. We reprove this result using Theorem \ref{thm:main}. 

\begin{theorem}[Kleinberg et al.~\cite{3colcluster}]\label{thm:clustering}
 For each $m>0$ there exists a planar graph that is not three-colorable with clustering $m$.
\end{theorem}

\begin{proof}
Take the point set $\Pc$ that realizes the non-three-colorable hypergraph $\Hc^3(m)$, and perturb the points so that no four lie on a circle. Consider the Delaunay graph of this point set, that is, we connect two points if they can be separated from the rest with a circle. From Theorem \ref{thm:main} we know that any three-coloring of this graph contains $m$ monochromatic points that can be separated from the rest by a circle and it is well-known that $m$ such points always form a connected component in the Delaunay graph, thus a monochromatic cluster in our case.
\end{proof}

\section{Open questions}\label{sec:6}
One of the most interesting questions left open is about unit disks.
Is there an $m$ such that every point set in the plane can be three-colored such that every unit disk containing exactly $m$ points contains at least two colors?

In fact, this conjecture has a strengthening in the dual, cover-decomposition problem.
Is there an $m$ such that every $m$-fold covering of $\Pc$ by disks can be partitioned into three parts such that any two parts cover $\Pc$?

If there is a counterexample to these questions, then it has to be quite different from ours, as already the two-color version \cite{unsplittable} needed the realization of a different hypergraph.

Coloring questions are also interesting in higher dimensions. Suppose we have a set of points and a family of balls in $\mathbb{R}^d$. How many colors are needed to color the points such that none of the balls that contain at least $m$ points are monochromatic? Theorem \ref{thm:main} can be extended to show that in some cases we need $d+2$ colors. On the other hand, no upper bound is known already for $d=3$ that would not depend on $m$.  

\subsubsection*{Acknowledgment}
We would like to thank Bal\'azs Keszegh for useful discussions and for reading the draft of this manuscript and also our anonymous reviewers for their valuable suggestions.


\small
\bibliographystyle{abbrv}

\end{document}